\documentclass[reqno,12pt]{amsart}

\usepackage{color} 
\usepackage{ifpdf}
\ifpdf
    \usepackage[pdftex]{graphicx}
    \usepackage[pdftex]{hyperref}
    \hypersetup{
        unicode=false,          
        pdftoolbar=true,        
        pdfmenubar=true,        
        pdffitwindow=false,     
        pdfstartview={FitH},    
        pdftitle={MCP Article},      
        pdfauthor={Michael Holst},   
        pdfsubject={Mathematics},    
        pdfcreator={Michael Holst},  
        pdfproducer={Michael Holst}, 
        pdfkeywords={PDE, analysis, mathematical physics}, 
        pdfnewwindow=true,      
        colorlinks=true,        
        linkcolor=red,          
        citecolor=blue,         
        filecolor=magenta,      
        urlcolor=cyan           
    }

\else
    \usepackage{graphicx}
    \usepackage{pstricks}
    
    \newcommand{\href}[2]{#2}
\fi

\usepackage{times}
\usepackage{amsfonts}

\usepackage{amsmath}
\usepackage{amsthm}
\usepackage{amssymb}
\usepackage{amsbsy}
\usepackage{amscd}

\usepackage{enumerate}
\usepackage{verbatim}
\usepackage{subfigure}
\usepackage[algo2e,linesnumbered,ruled]{algorithm2e}

\newtheorem{theorem}{Theorem}[section]
\newtheorem{corollary}[theorem]{Corollary}
\newtheorem{lemma}[theorem]{Lemma}

\newtheorem{remark}[theorem]{Remark}

\numberwithin{equation}{section}  

  \newcounter{mnote}
  \let\oldmarginpar\marginpar
    \renewcommand\marginpar[1]{\-\oldmarginpar[\raggedleft\footnotesize #1]%
    {\raggedright\footnotesize #1}}

\definecolor{myblue}{rgb}{0.2,0.2,0.7}
\definecolor{mygreen}{rgb}{0,0.6,0}
\definecolor{mycyan}{rgb}{0,0.6,0.6}
\definecolor{myred}{rgb}{0.9,0.2,0.2}
\definecolor{mymagenta}{rgb}{0.9,0.2,0.9}
\definecolor{mywhite}{rgb}{1.0,1.0,1.0}
\definecolor{myblack}{rgb}{0.0,0.0,0.0}

\newcommand{\beq}{\begin{equation}}
\newcommand{\eeq}{\end{equation}}
\newcommand{\beqa}{\begin{eqnarray}}
\newcommand{\eeqa}{\end{eqnarray}}
\newcommand{\tbar}{|\hspace*{-0.15em}|\hspace*{-0.15em}|}

\newcommand{\leqs}{\leqslant}      
     
\newcommand{\geqs}{\geqslant}      
      
\newcommand{\R}{{\mathbb R}}       

\newcommand{\cI}{{\mathcal I}}

\newcommand{\cT}{{\mathcal T}}

\begin{document}

\title[FE Error Estimates for Critical Growth Semilinear Problems]
      {Finite Element Error Estimates for Critical
       Growth Semilinear Problems without Angle Conditions}

\author[R. Bank]{Randolph E. Bank}
\email{rbank@ucsd.edu}

\author[M. Holst]{Michael Holst}
\email{mholst@math.ucsd.edu}

\author[R. Szypowski]{Ryan Szypowski}
\email{rszypows@math.ucsd.edu}

\author[Y. Zhu]{Yunrong Zhu}
\email{zhu@math.ucsd.edu}

\address{Department of Mathematics\\
         University of California San Diego\\ 
         La Jolla CA 92093}
\thanks{RB was supported in part by NSF Award 0915220.}
\thanks{MH was supported in part by NSF Awards~0715146 and 0915220,
and by DOD/DTRA Award HDTRA-09-1-0036.}
\thanks{RS and YZ were supported in part by NSF Award~0715146.}

\date{\today}

\begin{abstract}
In this article we consider {\em a priori} error and pointwise estimates 
for finite element approximations of solutions to 
semilinear elliptic boundary value problems in $d\geqs 2$ space dimensions,
with nonlinearities satisfying critical growth conditions.
It is well-understood how mesh geometry impacts finite element interpolant 
quality, and leads to the reasonable notion of {\em shape regular} simplex
meshes.
It is also well-known how to perform both mesh generation and simplex 
subdivision, in arbitrary space dimension, so as to guarantee the entire 
hierarchy of nested simplex meshes produced through subdivision
continue to satisfy shape regularity.
However, much more restrictive angle conditions are
needed for basic {\em a priori} quasi-optimal error 
estimates, as well as for {\em a priori} pointwise estimates.
These angle conditions, which are particularly difficult
to satisfy in three dimensions in any type of unstructured or adaptive setting,
are needed to gain pointwise control of the nonlinearity through discrete 
maximum principles.
This represents a major gap in finite element approximation theory
for nonlinear problems on unstructured meshes, and in particular for
adaptive methods.
In this article, we close this gap in the case of semilinear problems with
critical or sub-critical nonlinear growth, by deriving
{\em a priori} estimates directly, without requiring the discrete 
maximum principle, and hence eliminating the need for restrictive 
angle conditions.
Our main result is a type of local Lipschitz property that relies only
on the continuous maximum principle, together with the growth condition.
We also show that under some additional smoothness assumptions,
the {\em a priori} error estimate itself is enough to give
$L^{\infty}$ control the discrete solution, without the need for restrictive
angle conditions. Numerical experiments confirm our theoretical conclusions.
\end{abstract}

\maketitle


\tableofcontents

\clearpage

\section{Introduction}
   \label{sec:intro}

In this article we consider \emph{a priori} error estimates and discrete
pointwise estimates for Galerkin finite element approximation of solutions 
to a general class of semilinear problems 
satisfying certain growth conditions in $d$ space dimensions, which includes 
problems with critical and subcritical polynomial nonlinearity. 
When $d=2$ and $d=3$, it is well-understood how mesh geometry
impacts finite element interpolant quality (cf.~\cite{Babuska.I;Aziz.A1976}).
Such considerations lead to the requirement that simplex meshes used for 
finite element approximation satisfy a reasonable mesh condition known as 
{\em shape regularity}.
It is well-known how to perform both mesh generation and simplex subdivision,
in arbitrary space dimension, so as to guarantee that the entire 
hierarchy of nested simplex meshes produced through subdivision
satisfy shape regularity, and continue to do so 
asymptotically (cf.~\cite{AMP97,Bans91a,Bey95,Maub95}).
However, much more restrictive angle conditions are
needed for basic {\em a priori} quasi-optimal error 
estimates, as well as for {\em a priori} pointwise estimates
for Galerkin finite element approximations.
These angle conditions, which are particularly difficult
to satisfy in three dimensions in any type of unstructured or adaptive setting,
are needed to gain pointwise control of the nonlinearity through discrete 
maximum principles.
This represents a major gap in finite element approximation theory
for nonlinear problems on unstructured meshes, and in particular for
adaptive methods.
In this article, we close this gap in the case of semilinear problems with
critical or sub-critical nonlinear growth, by deriving
{\em a priori} estimates directly (both error estimates and 
discrete pointwise estimates), without requiring the discrete 
maximum principle, and hence eliminating the need for restrictive 
angle conditions.

Critical exponent problems arise in a fundamental way throughout
geometric analysis and general relativity.
One of the seminal critical exponent problems in nonlinear PDE
is the \emph{Yamabe Problem}~\cite{Aubi82}:
Find $u \in X$ (for some appropriate space $X$) such that
\begin{align}
-8\Delta_g u + R u & = R_u u^5 \quad \mbox{ in } \Omega,
\label{eqn-yamabe}
\\
u & > 0, 
\label{eqn-yamabe-bc}
\end{align}
where $\Omega$ is a Riemannian $3$-manifold,
$g$ is the positive definite metric on $\Omega$,
$\Delta_g$ is the Laplace-Beltrami operator generated by $g$,
$R$ is the scalar curvature of $g$,
and $R_u$ is the scalar curvature corresponding to the
\emph{conformally transformed} metric:
$
\overline{g} = \phi^4 g.
$
The coefficients $R$ and $R_u$ can take any sign.  The Banach space
$X$ containing the solution is an appropriate Sobolev class
$W^{s,p}(\Omega)$ for suitably chosen exponents $s$ and $p$.  If the
manifold $\Omega$ has a boundary, then boundary conditions are also
prescribed, such as $u=1$ on an exterior boundary to $\Omega$.  In the
case that ${\Omega \subset \mathbb{R}^3}$, and ${g_{ij}=\delta_{ij}}$,
then $\Delta_g$ reduces to just the Laplace operator on $\Omega$.
This problem is full of features that are challenging both 
mathematically and numerically, including:
critical exponent nonlinearity, 
potentially non-monotone nonlinearity,
spatial dimension $d \geqs 3$, 
and spatial domains that are typically non-flat Riemannian manifolds 
rather than simply open sets in $\mathbb{R}^d$.
A related critical exponent problem, containing all of the difficulties 
of~\eqref{eqn-yamabe}--\eqref{eqn-yamabe-bc},
plus the addition of low-order non-polynomial rational 
nonlinearities, arises in mathematical general relativity in the form of 
the {\em Hamiltonian constraint equation}; cf.~\cite{HNT07b}.

The presence of the term $u^5$ term in~\eqref{eqn-yamabe}, and in the
related Hamiltonian constraint in general relativity, is an example of 
a \emph{critical exponent problem} in space dimension three;
such problems are known to be difficult to analyze due to the loss of 
compactness of the embedding ${H^1 \subset L^{p+1}}$, where the 
dimension-dependent critical exponent ${p=(d+2)/(d-2)}$ takes
value $p=5$ when $d=3$.
Loss of compactness of the embedding creates obstacles that prevent
the use of compactness arguments in standard variational, Galerkin, 
and fixed-point techniques. 
As a result, these techniques are generally restricted to subcritical
nonlinearities, unless additional techniques give control
of the nonlinearity, such as \emph{a priori}
$L^{\infty}$, or \emph{pointwise}, control of solutions.
The inequality constraint~\eqref{eqn-yamabe-bc} creates
additional complexities in both theory and numerical treatment of
such problems, with only positive solutions having physical meaning.
Prior work on numerical methods for critical exponent
semilinear problems has focused primarily on the development of
adaptive methods for recovering solution blowup;
cf.~\cite{Budd.C;Humphries.A1998,Budd.C1998}.

The standard approach to obtaining \emph{a priori} $L^{\infty}$ bounds on
Galerkin approximations is to enforce approximation space properties
to guarantee discrete maximum principles, leading to
geometrical conditions on the underlying simplex mesh.
In the case of the $d=2$ Poisson problem, discrete maximum principles
can be established if all angles in the triangulation are 
non-obtuse (cf.~\cite{Ciarlet.P;Raviart.P1973}).
This was relaxed to ``summation of two opposite angles less 
or equal to $\pi$'' in~\cite[Page 78]{StFi73} (the so-called 
\emph{nonnegative triangulation} in~\cite{Draganescu.A;Dupont.T;Scott.L2004}).
In some cases discrete maximum principles
hold more generally~\cite{Santos.V1982}.
However, counter-examples indicate angle conditions cannot be relaxed 
as sufficient conditions~\cite{Draganescu.A;Dupont.T;Scott.L2004}.
For variable coefficients, anisotropic versions of non-obtuse angle 
conditions are required for discrete maximum principles
(cf.~\cite{Huang.W2010}).
The same angle conditions are needed in the 
nonlinear case~\cite{Kerkhoven.T;Jerome.J1990,Jungel.A;Unterreiter.A2005,Wang.J;Zhang.R2011,CHX06b}.
Due to the central role angle conditions play,
there is a growing literature on generating non-obtuse 
meshes~\cite{Krizek.M;Pradiova.J2000}.
Other approaches for obtaining $L^{\infty}$ estimates using local analysis 
include~\cite{Schatz.A;Wahlbin.L1978,Schatz.A1980}; 
related work on \emph{a priori} $L^{\infty}$ error estimates 
include~\cite{Bramble.J;Nitsche.J;Schatz.A1975,Scott.R1976}.
In~\cite{DLSW2010}, quasi-optimal $W^{1, \infty}$ error estimates are 
established under ``large-patch'' local quasi-uniformity conditions.
The proofs of the discrete maximum principle and the $L^{\infty}$ error 
estimates in the aforementioned works are quite technical, and have been 
limited to linear finite elements.
We note that~\cite{DLSW2010} has an extensive overview of
$L^{\infty}$ error estimates, and also relevant is~\cite{Deml07a} 
on localized pointwise (and negative norm) estimates for more general 
quasilinear problems.
Finally, we note other relevant semilinear work includes~\cite{CHX06b,HYZZ09a,Xu.J1996b,Xu.J;Zhou.A2000,Xu.J;Zhou.A2001a}.  

The need for angle conditions to gain pointwise control of the nonlinearity 
through discrete maximum principles represents a major gap in finite element 
approximation theory for nonlinear problems on unstructured meshes,
and is a particularly disturbing problem in the case of adaptive methods
that guarantee only shape regularity of meshes produced through subdivisions.
In this article, we close this gap in the case of semilinear problems with
critical or sub-critical nonlinear growth, by deriving
{\em a priori} estimates directly (both error estimates and 
discrete pointwise estimates), without requiring the discrete 
maximum principle, and hence eliminating the need for restrictive 
angle conditions.
Our main result is proving a type of local Lipschitz property 
for Galerkin finite element approximations for solutions to problems
with nonlinearities having critical and subcritical growth bounds,
using only \emph{a priori} $L^{\infty}$ control of the continuous solution, 
together with other results that are independent of the approximation space.
This result allows us to then establish, in successive order,
quasi-optimal \emph{a priori} energy error estimates for Galerkin
approximations, $L^2$ estimates via duality arguments, $L^{\infty}$
estimates via inverse-type inequalities, giving finally a discrete
$L^{\infty}$ bound without a discrete maximum principle, and therefore
without requiring angle conditions beyond shape regularity.
Although the techniques we use here are completely different,
our results on obtaining \emph{a priori} estimates without angle
conditions can be viewed as complementing the 
2006 work of Nochetto, Schmidt, Siebert, and 
Veeser~\cite{Nochetto.R;Schmidt.A;Siebert.K;Veeser.A2006}
on \emph{a posteriori} estimates without angle conditions,
for a similar class of monotone semilinear problems.
However, while some of our results require monotone nonlinearity, 
several results are established under weaker conditions
(see Assumptions~(A3$^{\prime}$) in Section~\ref{sec:pde}).

\emph{Outline of the paper.}
The remainder of the paper is structured as follows.
In Section~\ref{sec:pde} we describe a general class of semilinear 
problems, and under various assumptions
derive \emph{a priori} $L^{\infty}$ bounds for solutions
using cutoff functions and the De Giorgi iterative method 
(or Stampacchia truncation method).
In Section~\ref{sec:fem}, we develop quasi-optimal \emph{a priori}
error estimates for Galerkin approximations, where the nonlinearity is 
controlled only using a type of local Lipschitz property.
While the Lipschitz property is usually proved using discrete maximum
principles and $L^{\infty}$ control of the discrete solution, 
we establish this result for nonlinearities having critical and subcritical 
growth bounds using only \emph{a priori} $L^{\infty}$ control of the 
continuous solution, together with other results that are independent of the 
approximation space.
In Section~\ref{sec:disc-infty} we then use standard duality arguments 
to obtain corresponding $L^{2}$ error estimates.
Using inverse-type inequalities in the finite element approximation 
space, we then show that the discrete solution indeed has a uniform 
\emph{a priori} $L^{\infty}$ bound, without having access to the discrete 
maximum principle, and therefore without requiring restrictive angle 
conditions on the underlying finite element mesh.
Finally, in Section~\ref{sec:num} we examine the predictions made 
by the theoretical results through a sequence of numerical experiments.

\section{Semilinear Problems and \emph{A Priori} $L^{\infty}$ Estimates}
   \label{sec:pde}

In this section, we give an overview of a class of nonlinear elliptic boundary value problems on a bounded Lipschitz domain $\Omega \subset \R^{d}$ with $d=2$ or $d=3$. To begin with, we introduce some standard notation.  Given any subset $G\subset \R^{d},$ we use standard notation for the $L^{p}(G)$ spaces for $1\leqs p\leqs \infty,$ with the norm $\|\cdot \|_{0, p, G}$. We use standard notation for Sobolev norms $\|v\|_{k,2,G} =\|v\|_{H^{k}(G)}$ for the Sobolev space $H^{k}(G)$. For any function $v\in L^{p}(G)$ and $w\in L^{q}(G)$ with $p,q\geqs 1$ and $1/p + 1/q =1$, we denote the pairing $(v, w)_{G} := \int_{G}vw dx$. For simplicity, when $G = \Omega$, we omit if from the norms (or pairing). Given a function $g$ defined on $\Gamma=\partial \Omega,$ we define the affine space of $H^{1}(\Omega)$ as $H_{g}^{1}(\Omega) := \{v\in H^{1}(\Omega): v|_{\Gamma} = g\}.$ In particular, we have the following Poincar\'e-Sobolev inequality
\begin{equation}
	\label{eqn:poincare}
		\|u\|_{0,p} \leqs C_{s}(p) \|\nabla u\|_{0,2}, \quad \forall u\in H_{0}^{1}(\Omega),
\end{equation}
where $p<\infty$ if $d=2$ and $p = 2d/(d-2)$ if $d\geqs 3$, and the constant $C_{s}(p)$ depends only on $p$ and $\Omega$. In the sequel, we simply denote $C_{s}:=C_{s}(2)$ when $p=2$ in \eqref{eqn:poincare}.

We consider the following semilinear elliptic equation:
\begin{equation}
\label{eqn:model}
	-\nabla\cdot (D\nabla u)  +b(x, u) = f(x) \mbox{  in  } \Omega, \qquad u|_{\Gamma} =g,
\end{equation}
with the following assumptions:
\begin{enumerate}
	\item[(A1)] The diffusion tensor $D:\R^{d}\to \R^{d\times d} \in L^{\infty}(\Omega)$ satisfies that
$$
	m |\xi|^{2} \leqs \xi^{T} D \xi \leqs M |\xi|^{2}, \qquad \forall \xi \in \R^{d},
$$ 
for some constant $m, M>0.$ 
	\item[(A2)]  $f(x)\in L^{2}(\Omega)$ and $g \in L^{\infty}(\Gamma)$.
	\item[(A3)]  $b:\Omega\times\R\to \R$ is a Carath\'eodory function, i.e., for any given $\xi \in \R$ the function $b(\cdot, \xi):\Omega \to \R$  is measurable on $\Omega$, and for any given $x\in \Omega$ the function $b(x, \cdot):\R \to \R$ is smooth (cf. \cite[Definition 12.2]{FuKu80}). In the sequel, we will simply write $b(u)$ instead of $b(x, u)$, and assume that $b$ is monotone: 
		\begin{equation}
			\label{eqn:monotone}
			b'(\xi) \geqs 0, \quad \forall \xi\in \R.
		\end{equation}
		Without loss of generality, we also assume that $b(0) \equiv 0.$
	\item[(A4)]  $b$ satisfies the growth condition: there exists an integer $n$ with $1\leqs n \leqs \frac{d+2}{d-2}$ if $d\geqs 3$ and $1\leqs n <\infty$ if $d=2$ such that 
		\begin{equation}
		\label{eqn:growth}
			|b^{(n)} (\xi)|\leqs K, \quad \forall \xi\in \R,
		\end{equation}
		 for some constant $K>0.$ 
\end{enumerate}
The weak form of~\eqref{eqn:model} reads: Find $u\in H_{g}^{1}(\Omega)$ such that 
\begin{equation}
\label{eqn:weak}
	a(u, v) + (b(u), v) = (f,v), \qquad \forall v\in H_{0}^{1}(\Omega),
\end{equation}
where $a(u, v) := \int_{\Omega} D\nabla u \cdot\nabla v dx$.

Before moving on,
we make some brief comments about Assumptions~(A1)--(A4).
Assumption (A1) on the coefficient $D$ implies that the bilinear form $a(\cdot, \cdot)$ is coercive and continuous, namely, 
\begin{equation} 
	\label{eqn:coercive}
	m \|\nabla v\|_{0,2}^{2} \leqs a(v, v) \mbox{ and }  a(v, w)\leqs M \|\nabla v\|_{0,2} \|\nabla w\|_{0,2}, \quad \forall v,w\in H_{0}^{1}(\Omega).
\end{equation} 
This implies that the induced energy norm $\tbar v\tbar = \sqrt{a(v, v)}$ is equivalent to the $H^{1}$ semi-norm. The Assumption (A3) implies that 
\begin{equation}
	\label{eqn:wmonotone}
	(b(v) -b(w), v-w) \geqs 0, \quad \forall v, w\in H_{0}^{1}(\Omega).
\end{equation}
While a number of our results rely in the monotonicity Assumption~(A3),
we establish several key results under a weak condition
(see Assumption~(A3$^{\prime}$) below) that allows for
non-monotone nonlinearities.
Finally, note that Assumption (A4) holds when $b$ is a polynomial with 
degree up to (including) critical exponents. 
This assumption includes as examples the Yamabe problem, as well certain 
special cases of the Hamiltonian constraint in the Einstein equations 
mentioned in the introduction. 

In the remaining of this section, we try to establish \emph{a priori} $L^{\infty}$ bounds on the solution to~\eqref{eqn:weak} through maximum/minimum principles, which is quite standard in the PDE analysis (see for example~\cite{Gilbarg.D;Trudinger.N1983,Tayl96c}). Since it is important for our subsequent analysis, we include a proof of \emph{a priori} $L^{\infty}$ bounds on weak solutions using the de Giorgi iterative method (cf. ~\cite{De-Giorgi.E1957,Stampacchia.G1965}), which relies on the following lemma.
\begin{lemma}
	\label{lm:degiorgi}
	Let $\psi(\cdot)$ be a non-negative and non-increasing function on $[s_{0}, \infty)$ satisfying
	\begin{equation*}
		\psi(s) \leqs \left( \frac{A}{s-r}\right)^{\alpha}[\psi(r)]^{\beta},\qquad \forall s>r\geqs s_{0},
	\end{equation*}
	for some constant $A>0, \; \alpha >0$ and $\beta>1.$ Then 
	$$
		\psi(s) \equiv 0,\qquad \forall s\geqs s_{0}  +A2^{\beta/(\beta-1)}[\psi(s_{0})]^{(\beta-1)/\alpha}.
	$$
\end{lemma}
For a proof of this lemma, we refer to \cite[Lemma 4.1.1]{Wu.Z;Yin.J;Wang.C2006} or~\cite[Lemma 12.5]{Chipot.M2009}. 
By using this lemma, we are able to give explicit \emph{a priori} $L^{\infty}$ bound of the solution to~\eqref{eqn:weak}.
\begin{theorem}
\label{thm:cont-infty}
	Let the Assumptions (A1)-(A3) hold, and $u\in H_{g}^{1}(\Omega)$ be a weak solution to~\eqref{eqn:weak}. Then 
	\begin{equation}
		\label{eqn:cont-infty}
		\underline{u} \leqs u(x) \leqs \overline{u}, 	\quad a.e.\;\; x\in \Omega,
	\end{equation}
	where $\underline{u}$ and $\overline{u}$ are defined as
	\begin{eqnarray}
		&\underline{u} = \min\{0, {\rm ess} \inf_{x\in \partial \Omega} g(x)\} - C\|f\|_{0,2}, \label{eqn:barriers-under}\\
		&\overline{u} =\max\{0,  {\rm ess}\sup_{x\in \partial \Omega} g(x)\} + C \|f\|_{0,2}.\label{eqn:barriers-up}
	\end{eqnarray}
	Here the constant $C=  \frac{C^{2}_{s}(p)}{m} |\Omega|^{\frac{p-4}{2p}} 2^{\frac{p-2}{p-4}},$ where we choose $p>4$ (when $d=2$) or $p=6$ (when $d=3$), and $C_{s}(p)$ is the Poincar\'e-Sobolev constant in \eqref{eqn:poincare}.
\end{theorem}
\begin{proof}
	To prove the upper bound of \eqref{eqn:cont-infty}, let $ s_{0}  =\max \{0, {\rm ess} \sup_{x\in \partial \Omega} g(x) \}$ and define a test function 
	$$\phi^{+}(x) = (u(x)- r)^{+} := \max\{ u(x)-r, 0\}$$ with $r\geqs s_{0} \geqs 0$. Let $G(r):=\{ x\in \Omega: u(x) > r\}$. By the choice of $r$, it is obvious that $\phi^{+} \in H_{0}^{1}(\Omega)$, and it satisfies $$\nabla \phi^{+}(x) = \nabla u(x), \quad \mbox{ for a.e. } x\in G(r).$$
	
	By coercivity of $a(\cdot, \cdot)$, we have
	\begin{align*}
		m\int_{\Omega} |\nabla \phi^{+}|^{2} dx &\leqs a(\phi^{+}, \phi^{+}) =a(u, \phi^{+})\\
		& =  (f-b(u) ,\phi^{+})  \leqs (f,\phi^{+}),
	\end{align*}
	where in the last step we have used the assumption~\eqref{eqn:monotone} in (A3).
	Hence, 
	$$
		\|\nabla \phi^{+}\|_{0,2}^{2} \leqs \frac{1}{m} \int_{\Omega} |f\phi^{+}| dx.
	$$
	By Sobolev embedding theorem, $\|\phi^{+}\|_{0,p} \leqs C_{s}(p)\|\nabla \phi^{+}\|_{0,2}$ for $p\in (2, \infty)$ when $d=2$, or $p=6$ when $d = 3$, we obtain
	$$
		\|\phi^{+}\|_{0,p}^{2} \leqs \frac{C^{2}_{s}(p)}{m}\int_{\Omega} |f\phi^{+}|dx,
	$$
	where the Poincar\'e-Sobolev constant $C_{s}(p)$ depends only on the dimension $d$ and $\Omega$. 
	
	Notice that $\phi^{+}(x)\equiv 0$ when $x\in \Omega\setminus G(r),$ by H\"older inequality we obtain
	$$
		\|\phi^{+}\|_{0,p,G(r)}^{2} \leqs \frac{C^{2}_{s}(p)}{m}\int_{G(r)} |f\phi^{+}|dx\leqs \frac{C^{2}_{s}(p)}{m} |G(r)|^{\frac{1}{2} - \frac{1}{p}}\|\phi^{+}\|_{0,p, G(r)} \|f\|_{0,2, G(r)},
	$$
	 This implies
	\begin{equation}
	\label{eqn:right}
		\|\phi^{+}\|_{0,p, G(r)} \leqs \frac{C^{2}_{s}(p)}{m} |G(r)|^{\frac{1}{2} - \frac{1}{p}}\|f\|_{0,2}.
	\end{equation}
	Note that when $s>r,$ we have $G(s) \subset G(r)$ and $\phi^{+} \geqs s-r$ on $G(s).$ Therefore,
	\begin{equation}
	\label{eqn:left}
		\|\phi^{+}\|_{0,p,G(r)} \geqs \|\phi^{+}\|_{0,p,G(s)}  \geqs (s-r) |G(s)|^{1/p}.
	\end{equation}
	Combining the inequalities~\eqref{eqn:right} and~\eqref{eqn:left}, we obtain
	$$
		|G(s)|\leqs  \left(\frac{C^{2}_{s}(p)\|f\|_{0,2}}{m}\right)^{p} \frac{1}{(s-r)^{p}}|G(r)|^{\frac{p}{2} -1}.
	$$
	Now, by letting $\psi(s) = |G(s)|$,  $\alpha = p$, $\beta = p/2 -1>1$ and $A =C^{2}_{s}(p)\|f\|_{0,2}/m$ in Lemma~\ref{lm:degiorgi},  we obtain 
	$$|G(s)| \equiv 0, \quad \forall s\geqs s_{0} + \frac{C^{2}_{s}(p)\|f\|_{0,2}}{m} 2^{\frac{p-2}{p-4}}|G(s_{0} )|^{\frac{p-4}{2p}}.
	$$
	By definition of $G(s),$ this means that
	$$
		u \leqs  s_{0}  +  \frac{C^{2}_{s}(p)}{m} |\Omega|^{\frac{p-4}{2p}} 2^{\frac{p-2}{p-4}} \|f\|_{0,2} \quad \mbox{a.e. in } \Omega,
	$$
	which proves the upper bound.
	
	The proof of lower bound of \eqref{eqn:cont-infty} is similar. We define $ s_{0} :=  \max\{0,  -{\rm ess} \inf_{x\in \partial \Omega} g(x) \}$, and define the test function $\phi^{-} \in H_{0}^{1}(\Omega)$ as 
	$$
		\phi^{-}(x) := (u(x) + r)^{-} = \min\{ u(x) + r, 0\} \leqs 0,
	$$
	for some $ r\geqs  s_{0}  \geqs 0.$ Similarly, we introduce the subset $G(r):=\{ x\in \Omega: u(x) < -r\}.$ 
	Then by monotonicity \eqref{eqn:monotone} in (A3), one get that 
	\begin{align*}
		m\int_{\Omega} |\nabla \phi^{-}|^{2} dx &\leqs a(u, \phi^{-})=  (f - b(u) ,\phi^{-})  \leqs (f,\phi^{-}).
	\end{align*}
	The remaining of the proof are identical as the proof of upper bound, we omit the details here.
	This completes the proof.
\end{proof}

\begin{remark}
	There are obviously other methods for showing Theorem~\ref{thm:cont-infty}. Clearly, one of the main benefits of using the de Giorgi iterative Lemma~\ref{lm:degiorgi} is that it gives explicit bounds in the estimates. 
	Notice that  in Theorem~\ref{thm:cont-infty} we only use the conditions (A1)-(A3), so this theorem can be applied to a large class of nonlinear PDE problems, including the super-critical ones as for the regularized nonlinear Poisson-Boltzmann equation (cf.~\cite{CHX06b,HYZZ09a}). In the case of the Hamiltonian constraint application, we would have $b(u) =  a_{R} u + a_{\tau}u^{5} - a_{\rho}u^{-3} -a_{w}u^{-7}$. This definition of $b$ satisfies the condition if we assume that $a_{\tau}\geqs 0,$ $a_{\rho}\leqs 0$ and $a_{w}\leqs 0.$

	Finally, we note that the de Giorgi iterative argument can also applied to establish discrete maximum/minimum principles, which give rise to discrete \emph{a priori} $L^{\infty}$ bounds for the discrete solution; see~\cite{Jungel.A;Unterreiter.A2005,Wang.J;Zhang.R2011} for more detail. However, in the discrete setting,  it requires certain angle conditions in the underlying mesh in order to guarantee that the stiffness matrix is an M-matrix; this is what we wish to avoid in this paper, and will therefore take another approach in the following sections.
\end{remark}

We should also remark that the monotonicity assumption \eqref{eqn:monotone} is not essential for the maximum/minimum principles. In the remaining of this section, we should give another simple approach to show the $L^{\infty}$ bounds of the continuous solution with a slightly more general assumption on $b$. The following assumption on the nonlinearity allows for a class of functions containing both monotone and non-monotone cases:
\begin{enumerate}
\item[(A3$^{\prime}$)]
$b:\Omega\times\R\to \R$ is a Carath\'eodory function, which is barrier monotone in its second argument: there exist constants 
$\alpha, \beta \in \mathbb{R}$, with ${\alpha \leqs \beta}$, such that
\begin{align*}
b(x,s) - f(x) &\geqs 0, \quad \forall s \geqs \beta, 
\quad \mbox{ a.e. in } \Omega.
\\
b(x,s) -f(x) &\leqs 0, \quad \forall s \leqs \alpha, 
\quad \mbox{ a.e. in } \Omega.
\end{align*}
\end{enumerate}
We have the following theorem based on the Assumptions (A1), A(2) and (A3$^{\prime}$):
\begin{theorem}[\emph{A Priori} $L^{\infty}$ Bounds]
\label{thm:gcont-infty}
Let the Assumptions (A1)-A(2) and (A3$^{\prime}$) hold.
Let $u\in H_g^1(\Omega)$ be any
weak solution to~\eqref{eqn:weak}.
Then  
\begin{equation}
\label{eqn:gcont-infty}
\underline{u} \leqs u \leqs \overline{u}, \quad \mbox{ a.e. in } \Omega,
\end{equation}
for the constants $\overline{u}$ and $\underline{u}$ defined by
\begin{equation}
\label{eqn:barriers}
	\overline{u} := \max\left\{\beta, \sup_{x\in \partial \Omega} g(x)\right\},\qquad \underline{u}: = \min\left\{\alpha, \inf_{x\in \partial \Omega} g(x)\right\},
\end{equation} 
where $\alpha \leqs \beta$ are the constants in Assumption (A3$^{\prime}$). 
\end{theorem}
\begin{proof}
To prove the upper bound, let us introduce
$$
\phi = (u-\overline{u})^+=\max\{ u-\overline{u}, 0\}.
$$
By the definition of $\overline{u}$, 
it follows (cf.~\cite[Theorem 10.3.8]{StHo2011a})
that $\phi \in H_{0}^{1}(\Omega)$ and $\phi \geqs 0$ a.e. in $\Omega$. 
Taking $v= \phi$ in \eqref{eqn:weak}, we have 
\begin{eqnarray*}
	a(u, \phi) = a(u-\overline{u}, \phi) = a(\phi, \phi) \geqs m\|\nabla \phi\|_{0,2}^{2}.
\end{eqnarray*}
This implies that 
$$
	m\|\nabla \phi\|_{0,2}^{2} \leqs a(u, \phi) = (f-b(u), \phi) \leqs 0,
$$
since $f-b(u) \leqs 0 $ a.e. in the support of $\phi$, that is, $\|\nabla \phi\|_{0,2}\equiv 0$ which yields
$\phi =0$. Therefore, the upper bound of \eqref{eqn:gcont-infty} holds.

Similarly, we introduce 
$$
	\psi = (u-\underline{u})^-=\min\{ u-\underline{u}, 0\}.
$$
It is obvious that $\psi \in H_{0}^{1}(\Omega)$ can be used as a test function in \eqref{eqn:weak}.
Moreover, $\psi \leqs 0$ a. e. in $\Omega$, and Assumption (A3$^{\prime}$) implies
$f - b(u) \geqs 0$ on the support of $\psi$. Therefore,
$$
	m\|\nabla \psi\|_{0,2}^{2} \leqs a(u, \psi) =  (f-b(u), \psi) \leqs 0,
$$
which implies $\psi \equiv 0$ as before. This proves the lower bound of \eqref{eqn:gcont-infty}.
\end{proof}

The same technique in Theorem~\ref{thm:gcont-infty} can be applied in the discrete setting,  again with additional assumption on the mesh. In fact, if the triangulation satisfies that 
$$a(\phi_{i}, \phi_{j})\leqs 0, \qquad \forall i\neq j$$
where $\phi_{i}$ and $\phi_{j}$ are the basis functions corresponding to the vertices $i$ and $j$ respectively, then the conclusion of Theorem~\ref{thm:gcont-infty} still holds for the finite element solution $u_{h}$. However, this is out of the scope of this paper. For more details, we refer to \cite{HSZ10a}.

\section{Quasi-optimal Estimates without Angle Conditions}
   \label{sec:fem}

In this section, we consider the finite element approximation of~\eqref{eqn:weak} and derive quasi-optimal error estimates without angle conditions of any type.
Without loss of generality, we assume that the Dirichlet data satisfies $g \equiv 0$ for ease of exposition.
Let $\cT_{h}$ be a quasi-uniform triangulation of $\Omega$. We emphasize that the triangulation does {\em not} require any particular angle conditions other than the quasi-uniformity. 
Let $V_{h} = \{ v\in H_{0}^{1}(\Omega): v|_{T} \in \mathbb{P}_{k}(T), \;\; \forall T\in \cT_{h}\}$ be the standard finite element space defined on $\cT_{h}$, where $\mathbb{P}_{k}(T)$ ($k\geqs 1$) is the space of polynomials of degree $\leqs k$ define on $T$. The finite element discretization of~\eqref{eqn:weak} reads: Find $u_{h} \in V_{h}$ such that 
\begin{equation}
	\label{eqn:disc}
	        a(u_{h}, v) + (b(u_{h}), v) = (f, v),\qquad \forall v\in V_{h}.
\end{equation}
We remark that here we do not require the finite element space to be piecewise linear, which is a requirement in most literature for discrete maximum/minimum principles. 

We now give a simple lemma that establishes a \emph{a priori} energy bounds
on solutions to~\eqref{eqn:weak} and~\eqref{eqn:disc} that are independent
of most features of the problem; these bounds will be critical for proving
quasi-optimal error estimates without mesh conditions.
\begin{lemma}
\label{lm:exist}
Let the Assumptions (A1)-(A2) hold, and $u \in H_0^1(\Omega)$ and
$u_h \in V_h \subset H_0^1(\Omega)$ be the weak solutions 
to~\eqref{eqn:weak} and~\eqref{eqn:disc}, respectively. If the nonlinear 
function $b$ satisfies (A3)
then the following \emph{a priori} energy bounds hold:
\begin{align}
\|\nabla u \|_{0,2} &\leqs  \frac{C_{s}}{m}\|f\|_{0,2},
   \label{eqn:apriori-energy} 
\\
\|\nabla u_h \|_{0,2} &\leqs \frac{C_{s}}{m}\|f\|_{0,2},
   \label{eqn:apriori-energy-disc}
\end{align}
where $C_{s}$ is the Poincar\'e-Sobolev constant in \eqref{eqn:poincare}. 
\end{lemma}
\begin{proof}
	If $b$ satisfies Assumption (A3), then $(b(u), u) \geqs 0$ by \eqref{eqn:monotone}, which implies that 
	\begin{eqnarray*}
		\tbar u\tbar ^{2} = a(u, u) = (f-b(u), u)\leqs \|f\|_{0,2} \|u\|_{0,2}.
	\end{eqnarray*}
	Therefore, the coercivity \eqref{eqn:coercive} of $a(\cdot, \cdot)$ and Poincar\'e inequality \eqref{eqn:poincare} imply that
	$$
		\|\nabla u \|_{0,2} \leqs \frac{C_{s}}{m} \|f\|_{0,2},
	$$
	which shows the inequality \eqref{eqn:apriori-energy}.  The inequality \eqref{eqn:apriori-energy-disc} follows by the same arguments. 	
\end{proof}

With a certain convenient assumption \eqref{eqn:lipschitz} on the nonlinear function $b$ that
we will examine in more detail shortly, we can easily obtain the following 
quasi-optimal \emph{a priori} error estimate for Galerkin approximations.
\begin{theorem}
        \label{thm:err}
	Let the Assumptions (A1)-(A3) hold, and $u$ and $u_{h}$ be the solutions to~\eqref{eqn:weak} and~\eqref{eqn:disc} respectively.  Assume that there exists a constant $C_{L}$ such that 
	\begin{equation}
		\label{eqn:lipschitz}
			(b(u) - b(u_{h}), v) \leqs C_{L} \|\nabla(u -u_{h})\|_{0,2} \|\nabla v\|_{0,2}, \qquad \forall  v \in H_{0}^{1}(\Omega).
	\end{equation} 
	Then we have the following quasi-optimal error estimate:
$$
        \tbar u - u_{h}\tbar \leqs \left(1+ \frac{C_{L}}{m}\right)\inf_{v\in V_{h}} \tbar u - v\tbar.
$$
\end{theorem}
\begin{proof}
	By subtracting equation~\eqref{eqn:disc} from~\eqref{eqn:weak}, we have 
\begin{equation*}
        a(u-u_{h}, v) + (b(u) - b(u_{h}), v) =0, \qquad \forall v\in V_{h}.
\end{equation*}
By using this identity, we obtain
\begin{align*}
\tbar u - u_{h}\tbar^{2}
  &= a(u-u_{h}, u-u_{h}) \\
  &= a(u-u_{h}, u-v) + a(u-u_{h}, v-u_{h}), \quad \forall v\in V_{h}\\
  &= a(u-u_{h}, u-v) + (b(u) - b(u_{h}), u_{h} -v)\\
  &= a(u-u_{h}, u-v) + (b(u) - b(u_{h}), u -v) - (b(u) - b(u_{h}), u - u_{h}),\\
  &\leqs \tbar u-u_{h}\tbar  \tbar u - v\tbar + (b(u) - b(u_{h}), u-v),
\end{align*}
where we used the monotonicity~\eqref{eqn:monotone} of $b$ in Assumption (A3).

Now by assumption~\eqref{eqn:lipschitz} and coercivity \eqref{eqn:coercive}, we have 
$$
	(b(u) -b(u_{h}), u-v) \leqs C_{L}\|\nabla(u -u_{h})\|_{0,2} \|\nabla(u-v)\|_{0,2} \leqs \frac{C_{L}}{m}\tbar u - u_{h}\tbar \tbar u -v\tbar.
$$
The conclusion then follows since $v\in V_{h}$ is arbitrary.
\end{proof}
\begin{remark}
We note that the monotonicity assumption~\eqref{eqn:monotone} in 
Assumption~(A3) as used in Theorem~\ref{thm:err} may be weakened 
in several different ways.
Such a generalization of Theorem~\ref{thm:err}, using an
argument based on $L^{2}$-lifting, can be found in \cite{HSTZ10a}.
\end{remark}

Theorem~\ref{thm:err} provides a general framework for establishing quasi-optimal \emph{a priori} error estimates for~\eqref{eqn:model}.
The key is to realize the assumption~\eqref{eqn:lipschitz}, which is a relaxation of the Lipschitz continuity of $b$.
For nonlinearities that are not Lipschitz continuous, a standard approach to deriving inequality~\eqref{eqn:lipschitz} is to use continuous and discrete $L^{\infty}$ bounds on $u$ and $u_{h}$, as was done in~\cite{CHX06b,HYZZ09a} for the Poisson-Boltzmann equation.
In this approach, since $\|u\|_{0,\infty} \leqs M_{1}$ and $\|u_{h}\|_{0,\infty}\leqs M_{2}$, one can control the nonlinear term as 
\begin{align*}
	b(u) - b(u_{h})&= b'(u_{h} + t(u -u_{h})) (u -u_{h})\\
	&\leqs \|b'(u_{h} + t (u-u_{h}))\|_{0, \infty} \|u - u_{h}\|_{0,2}, 
\end{align*}
for some $ t\in [0,1]$.
In this way, one can easily obtain~\eqref{eqn:lipschitz}.
Unfortunately, this approach requires \emph{a priori} $L^{\infty}$ bounds on $u_{h}$.
The standard approach for obtaining such \emph{a priori} $L^{\infty}$ bounds on $u_{h}$ is by discrete maximum principles, which requires restrictive angle conditions.
These angle conditions are particularly difficult to satisfy in the unstructured and adaptive settings, especially in three space dimensions.

However, with the help of the growth condition (A4), it is actually possible to establish assumption~\eqref{eqn:lipschitz} without employing discrete $L^{\infty}$ bounds on $u_{h}$, and hence without any assumptions on the mesh at all. The remainder of this section is devoted to proving this.
We first note that Lemma~\ref{lm:exist} gives both
continuous and discrete \emph{a priori} energy bounds that depend only on $\|f\|_{0,2}$, and on the coercivity constant $m$ and the Poincar\'e-Sobolev constant $C_{s}$. In particular, there is no dependence on the discretization parameter $h$ in the case of the bound for $u_h$.
Hence, combining~\eqref{eqn:apriori-energy}
and~\eqref{eqn:apriori-energy-disc} 
with the triangle inequality and coercivity of $a(\cdot, \cdot)$,
we obtain
\begin{equation}
	\label{eqn:ball}
		\| \nabla(u - u_{h})\|_{0,2}  \leqs  \frac{2C_{s}}{m} \|f\|_{0,2} =:R,
\end{equation}
where $R = R(C_{s}, m, \|f\|_{0,2})$ is a constant independent of $h$.
This observation makes possible the following local Lipschitz result.
\begin{theorem}
\label{thm:nonlinear}
Let the Assumptions (A1)-(A4) hold, and let $u$ and $u_{h}$ be the solutions to~\eqref{eqn:weak} and~\eqref{eqn:disc}, respectively.  
Then
\begin{equation}
\label{E:lipschitz}
(b(u) - b(u_{h}), v) \leqs C_{L} \| \nabla(u - u_h) \|_{0,2} \|\nabla v \|_{0,2},
\quad \forall v \in H_{0}^{1}(\Omega),
\end{equation}
where $C_{L}=C_{L}(\Omega,f,\|u\|_{\infty},d,n, m)$ is independent of $h$.
\end{theorem}
\begin{proof}
We begin with the H\"{o}lder inequality
\begin{equation}
\label{E:holder}
(b(u) - b(u_{h}), v) \leqs
\left\| b(u) - b(u_{h})\right\|_{0,p}\|v\|_{0,p^*},
\end{equation}
where $1/p+1/p^*=1$.
For $d\geqs 3$, we take $p^*=2d/(d-2)$ and $p = p^{*}/(p^{*}-1)$. The choice of $p^{*}$ allows us to use the Poincar\'e-Sobolev inequality \eqref{eqn:poincare} for $\|v\|_{0, p^{*}}$. 
For example, when $d=3$, we take $p^*=6$ and $p=6/5$.
For $d=2$, we may take any $1<p^* < \infty$ and $p=p^{*}/(p^{*}-1)$.

Notice that by Taylor expansion, we can write $b(u) - b(u_{h})$ as the finite sum
\begin{equation}
	\label{eqn:taylor}
b(u) - b(u_{h}) = \sum_{k=1}^{n-1} \frac{1}{k!} b^{(k)} (u) (u-u_{h})^{k} + \frac{1}{n!} b^{(n)}(\xi) (u-u_{h})^{n},
\end{equation}
for some $\xi \in H_{0}^{1}(\Omega).$ 
By Theorem~\ref{thm:cont-infty} , the Assumptions (A1)-A(3) implies the \emph{a priori} $L^{\infty}$ bound of $u$ \eqref{eqn:cont-infty}, hence $\|b^{(k)}(u)\|_{0,\infty}$ is bounded for $k=1, 2, \cdots, n-1$. On the other hand, we have $\|b^{(n)}(\xi)\|_{0,\infty} <K $ by Assumption (A4). Therefore, by Minkowski inequality and \eqref{eqn:taylor} we obtain
\begin{align}
\|b(u) - b(u_{h})\|_{0,p}
&\leqs  
        C_{1}\sum_{k=1}^{n} \|(u-u_{h})^{k}\|_{0,p}
 =  C_{1}
      \left( \sum_{k=1}^{n} \|u-u_{h}\|_{0,kp}^k \right),
\label{E:taylor}
\end{align}
where $C_{1}:= C_1(\Omega,b,\underline{u}, \overline{u}, K, n)$ is a constant independent of $h.$
Our range of $k$ and choice of $p$ ensures that
$1 \leqs kp \leqs np \leqs p^{*}$
when $d\geqs 3$, and $1 \leqs kp \leqs np < \infty$ when $d=2$.
Therefore, by \eqref{eqn:poincare} we have
\begin{equation}
\label{E:sobolev}
\|u-u_h\|_{0,kp} \leqs C_s(kp) \|\nabla (u-u_h)\|_{0,2}, \quad 1 \leqs k \leqs n.
\end{equation}
Using this together with~\eqref{eqn:ball}, we have
\begin{align}
\sum_{k=1}^{n} \|u-u_{h}\|_{0,kp}^k
&\leqs \sum_{k=1}^{n} C_{s}^{k}(kp) \|\nabla(u-u_{h})\|_{0,2}^k 
\nonumber \\
&\leqs \max_{1\leqs l \leqs n} C_{s}^{l}(lp) \left( \sum_{k=1}^{n} \|\nabla(u-u_{h})\|_{0,2}^k \right) 
\nonumber \\
&=     \max_{1\leqs l \leqs n} C_{s}^{l}(lp)
        \left( \frac{1-\|\nabla(u-u_{h})\|_{0,2}^{n}}{1-\| \nabla(u-u_{h}) \|_{0,2}} \right) 
        \| \nabla(u-u_{h}) \|_{0,2}
\nonumber \\
&\leqs C_2(\Omega,n,R)\| \nabla(u-u_{h}) \|_{0,2},
\label{E:polybound}
\end{align}
where $R$ is the constant defined~\eqref{eqn:ball}. Here we assumed $R\neq 1$ without loss of generality. Finally, using Sobolev inequality on $\|v\|_{p^{*}}$ and Combining~\eqref{E:holder}, \eqref{E:taylor}, and~\eqref{E:polybound} now
gives~\eqref{E:lipschitz} with $C_{L}=C_1 C_3 C_{s}(p^{*})$.
\end{proof}

As a result of Theorem ~\ref{thm:err} and the standard interpolation error (see~\eqref{eqn:inter}), we obtain the following quasi-optimal error estimate:
\begin{corollary}
	\label{cor:quasi-opt}
	Let the Assumptions (A1)-(A4) hold, and $u$ and $u_{h}$ be the solutions to~\eqref{eqn:weak} and~\eqref{eqn:disc}, respectively.  If in addition $u\in H^{s}(\Omega)\cap H_{0}^{1}(\Omega)$ for some $s>1$, then
	\begin{equation}
		\label{eqn:quasi-opt}
			\tbar u - u_{h} \tbar \leqs C h^{s-1} \|u\|_{s, 2},
	\end{equation}
	where $C$ is a constant independent of $h.$
\end{corollary}

\section{Discrete $L^{\infty}$ Error Estimates without Angle Conditions}
   \label{sec:disc-infty}

Once we obtain the quasi-optimal error estimate in Theorem~\ref{thm:err}, we can use it to obtain the error estimates in other norms, such as $L^{2}$ and $L^{\infty}$ estimates.
In this section, we first derive an $L^{2}$-error estimate using the standard Aubin-Nitsche technique.
This $L^{2}$-error estimate is not only of its own interest, but also has several applications.
For example, we can use it to obtain some $L^{\infty}$ error estimates using certain inverse-type inequalities.
The $L^{\infty}$-error estimate will then subsequently be useful in establishing \emph{a priori} $L^{\infty}$ bounds for the discrete solution, without any type of angle conditions.
In other words, rather than first imposing restrictive angle conditions to get a discrete maximum principle, and using this to control the 
nonlinearity to get a \emph{a priori} error estimate, we essentially turn things around and use the error estimates to establish \emph{a priori} $L^{\infty}$ estimates of the discrete solution $u_{h}$.

We assume that the solution $u$ to \eqref{eqn:weak} satisfies the regularity $u\in H^{s}(\Omega) \cap H_{0}^{1}(\Omega)$ for some $s >1$.
Recall that given the quasi-uniform triangulation $\cT_{h}$ with the mesh-size $h$, there exists an interpolation operator ${\cI}_{h}: H^{s}(\Omega)\cap H_{0}^{1}(\Omega) \to V_{h}$ such that the following standard interpolation error estimates hold (cf.~\cite[Chapter 3]{Ciarlet.P1978} or~\cite{Scott.R;Zhang.S1990}):
\begin{equation}
\label{eqn:inter}
	\|v - \cI_{h} v\|_{ s_{0} ,2} \leqs C h^{s- s_{0} } \|v\|_{s, 2}, \qquad  s_{0}  = 0 ,1,
\end{equation}
and 
\begin{equation}
\label{eqn:Linfty-int}
	\|v - \cI_{h} v\|_{0,\infty} \leqs C h^{s-d/2} \|v\|_{s, 2},
\end{equation}
for any $v\in H^{s}(\Omega)\cap H_{0}^{1}(\Omega).$ In fact, \eqref{eqn:inter} has been used to obtain \eqref{eqn:quasi-opt} in Corollary~\ref{cor:quasi-opt}.

We now derive the $L^{2}$ error estimate for $u-u_{h}$ by duality argument. 
To begin, let us introduce the following linear adjoint problem: Find $w\in H_{0}^{1}(\Omega)$ such that
\begin{equation}
	\label{eqn:dual}
	 a(v, w) + (b'(u) v, w) = (u - u_{h}, v), 	\qquad v \in H_{0}^{1}(\Omega).
\end{equation}
We assume that the linear problem~\eqref{eqn:dual} has the regularity 
\begin{equation}
\label{eqn:reg}
	\|w\|_{t,2} \leqs C_{r} \|u - u_{h}\|_{0,2}
\end{equation}
for some $t >1.$ Then we have the following $L^{2}$ error estimate for $u_{h}$:
\begin{theorem}[$L^{2}$ Error Estimate]
	\label{thm:l2err}
	Let Assumptions (A1)-(A4) hold, and $u \in H^{s}(\Omega)\cap H_{0}^{1}(\Omega)$ with $s >1$ be the solution to~\eqref{eqn:weak},  and $u_{h}$ be the solution to~\eqref{eqn:disc}. Suppose the dual problem \eqref{eqn:dual} satisfies the regularity assumption \eqref{eqn:reg} with $t>1$. Then
	\begin{equation}
	\label{eqn:dualest}
		\|u - u_{h}\|_{0,2} \leqs C \left(h^{s+t -2}  + h^{2(s-1)}\right) \|u\|_{s,2},
	\end{equation}
	where $C$ is independent of $h.$
\end{theorem}
\begin{proof}
As in the proof of Theorem~\ref{thm:nonlinear}, recall the Taylor formula \eqref{eqn:taylor} for $b(u) -b(v):$
	\begin{equation}
	\label{eqn:dtaylor}
		b(u) - b(u_{h}) = \sum_{k=1}^{n-1} \frac{1}{k!} b^{(k)}(u) (u-u_{h})^{k} + \frac{1}{n!} b^{(n)}(\xi) (u-u_{h})^{n},
	\end{equation}
for some $\xi \in H_{0}^{1}(\Omega)$. 
	We have that $\|b^{(k)} (u)\|_{0,\infty}$ is bounded for any $k=1, \cdots, n-1$ since $\|u\|_{0,\infty}$ is bounded by Theorem~\ref{thm:cont-infty}. Moreover, $\|b^{(n)}(\xi)\|_{0,\infty} \leqs K$ for some constant $K>0$ as stated in Assumption (A4).  In particular, when $n =1$, we may treat the problem as linear case, and the conclusion follows by the standard duality argument. Therefore, without loss of generality we assume $n\geqs 2$ in the following proof. 
		
Taking $v = u -u_{h}$ in~\eqref{eqn:dual} we obtain that
\begin{align}
	\|u - u_{h}\|^{2}_{0,2} &= a(u-u_{h}, w) + (b'(u) (u-u_{h}), w)\nonumber\\
	&= a(u-u_{h}, w- w_{h}) + (b'(u)(u-u_{h}), w-w_{h}) \nonumber\\
	& \qquad + a(u-u_{h}, w_{h}) + (b'(u) (u-u_{h}), w_{h}).\label{eqn:dualA}
\end{align}
Since $u_{h}\in V_{h}$ is the solution to the discrete semilinear problem~\eqref{eqn:disc}, we have
\begin{align*}
	0&= a(u-u_{h}, w_{h}) + (b(u)-b(u_{h}), w_{h}).
\end{align*}
By \eqref{eqn:dtaylor}, the last two terms in \eqref{eqn:dualA} can be written as
\begin{eqnarray*}
	&&a(u-u_{h}, w_{h}) + (b'(u) (u-u_{h}), w_{h})\\ 
	&&= -\sum_{k=2}^{n-1}\frac{1}{k!} (b^{(k)}(u) (u-u_{h})^{k}, w_{h}) - \frac{1}{n!} (b^{(n)}(\xi) (u-u_{h})^{n}, w_{h}).
\end{eqnarray*}
Hence, by H\"older inequality, the Poincar\'e-Sobolev inequality \eqref{eqn:poincare} and the fact that $\|b^{(k)} (u)\|_{0,\infty}\;\; (k=1, \cdots, n-1)$ and $\|b^{(n)}(\xi)\|_{0,\infty}$ are uniformly bounded,  we obtain
\begin{align}
	\|u - u_{h}\|^{2}_{0,2} &\leqs \tbar u-u_{h}\tbar \tbar w-w_{h}\tbar + \|b'(u)\|_{0,\infty} \|u-u_{h}\|_{0,2} \|w-w_{h}\|_{0,2} \nonumber\\
	& \qquad - \sum_{k=2}^{n-1}\frac{1}{k!} (b^{(k)}(u) (u-u_{h})^{k}, w_{h})  -\frac{1}{n!} (b^{(n)}(\xi) (u-u_{h})^{n}, w_{h})\nonumber\\
	&\lesssim \tbar u-u_{h}\tbar \tbar w-w_{h}\tbar + \sum_{k=2}^{n} \|u-u_{h}\|_{0,p^{*}}^{k} \|w_{h}\|_{0,p^{*}/(p^{*}-k)}\nonumber\\
	&\lesssim \tbar u-u_{h}\tbar \tbar w-w_{h}\tbar +  \sum_{k=2}^{n} \tbar u-u_{h}\tbar^{k} \|w_{h}\|_{0,p^{*}/(p^{*}-k)}, \label{eqn:dualB}
\end{align}
where we choose $p^{*} = 6$ for $d=3$ and $p^{*} >n$ when $d=2.$ 

		To estimate the right hand side of \eqref{eqn:dualB}, we take $w_{h} =P_{h}w\in V_{h}$ as the Galerkin projection of $w$, that is, $w_{h}$ is the finite element solution to~\eqref{eqn:dual} on $V_{h}$. Then by the standard finite element approximation property for the linear equation~\eqref{eqn:dual}, we have 
\begin{equation}
\label{eqn:femlinear}
	\tbar w -w_{h}\tbar \lesssim h^{t-1} \|w\|_{t,2} \lesssim h^{t-1} \|u-u_{h}\|_{0,2},
\end{equation}
where in the second inequality, we used the regularity assumption~\eqref{eqn:reg}. Therefore, combining \eqref{eqn:femlinear} with \eqref{eqn:quasi-opt} in Corollary~\ref{cor:quasi-opt}, we obtain that 
\begin{equation}
	\label{eqn:dualC}
\tbar u-u_{h}\tbar \tbar w-w_{h}\tbar \lesssim h^{s+t-2} \|u\|_{s,2} \|u-u_{h}\|_{0,2}.
\end{equation}
Now we turn to estimate the second term  $\sum_{k=2}^{n} \tbar u-u_{h}\tbar^{k} \|w_{h}\|_{0,p^{*}/(p^{*}-k)}$ in \eqref{eqn:dualB}. First of all, by the Poncar\'e-Sobolev inequality \eqref{eqn:poincare}, we have 
$$
	\|w_{h}\|_{0,p^{*}/(p^{*}-k)} \leqs C_{s}\left(\frac{p^{*}}{p^{*}-k}\right) \|\nabla w_{h}\|_{0,2}, \quad k = 2, 3, \cdots, n.
$$
By~\eqref{eqn:monotone} in Assumption (A4) on $b$, a similar argument as in Lemma~\ref{lm:exist} yields
\begin{align*}
	m\|\nabla w_{h}\|_{0,2}^{2} &\leqs\tbar w_{h} \tbar^{2} = a(w_{h}, w_{h})\\
	 &\leqs a(w_{h}, w_{h}) + (b'(u) w_{h}, w_{h}) = (u -u _{h}, w_{h}) \\
	 &\leqs \|u-u_{h}\|_{0,2} \|w_{h}\|_{0,2} \leqs C_{s} \|u-u_{h}\|_{0,2} \|\nabla w_{h} \|_{0,2},
\end{align*}
which gives us the estimate
$$
	\|\nabla w_{h}\|_{0,2} \leqs \frac{C_{s}}{m} \|u - u_{h}\|_{0,2}.
$$
Therefore, we obtain that 
$$
	\|w_{h}\|_{0, p^{*}/(p^{*}-k)} \lesssim \|u-u_{h}\|_{0,2}, \quad \forall k=2, \cdots, n. 
$$
Thus, we have 
\begin{equation}
	\label{eqn:dualD}
	\sum_{k=2}^{n} \tbar u-u_{h}\tbar^{k} \|w_{h}\|_{0,p^{*}/(p^{*}-k)} \lesssim  \|u-u_{h}\|_{0,2}\sum_{k=2}^{n} h^{k(s-1)}\| u \|^{k}_{s ,2}.
\end{equation}
Combining inequalities \eqref{eqn:dualB}, \eqref{eqn:dualC} and \eqref{eqn:dualD}, the inequality \eqref{eqn:dualest} then follows. 
\end{proof}
\begin{remark}
	In case of full regularity, namely $s=t=2$, then we have the optimal $L^{2}$ error estimate. 
	\begin{equation}
		\label{eqn:l2full}
		\|u - u_{h}\|_{0,2} \leqs C h^{2}.
	\end{equation}
\end{remark}

We now try to give a simple $L^{\infty}$ error estimate.
We first give the following general lemma. 
\begin{lemma}[$L^{\infty}-L^2$ Lemma]
\label{lm:Linfty}
	Let $u$ and $u_{h}$ be the solutions to~\eqref{eqn:weak} and~\eqref{eqn:disc}, respectively. Then we have 
	\begin{equation}
		\| u - u_{h}\|_{0,\infty} \lesssim \inf_{v_{h} \in V_{h}} (\|u - v_{h}\|_{0,\infty} + h^{-d/2}\|u - v_{h}\|_{0,2}) + h^{-d/2}\|u - u_{h}\|_{0,2}.
	\end{equation}
\end{lemma}
\begin{proof}
	We refer to~\cite[Remark 6.2.3]{Quarteroni.A;Valli.A2008} for a proof of this lemma. 
\end{proof}

By the interpolation error estimates~\eqref{eqn:inter} and~\eqref{eqn:Linfty-int}, we have 
\begin{equation}
	\label{eqn:linfty0}
	\inf_{v_{h} \in V_{h}} (\|u - v_{h}\|_{0,\infty} + h^{-d/2}\|u - v_{h}\|_{0,2}) \lesssim h^{s-d/2} \|u\|_{s,2}.
\end{equation}
The above discussion together with the  $L^{2}$ error estimate~\eqref{eqn:dualest} in Theorem~\ref{thm:l2err} leads to the following $L^{\infty}$ error as well as the discrete $L^{\infty}$ bound:
\begin{corollary}[$L^{\infty}$ Error Estimate and \emph{A Priori} $L^{\infty}$ Bound]
\label{cor:linfty}
	Let Assumptions (A1)-(A4) hold, and $u\in H^{s}(\Omega)\cap H_{0}^{1}(\Omega)$ with $s>1$ be the solution to~\eqref{eqn:weak}, and $u_{h}\in V_{h}$ is the solution to~\eqref{eqn:disc}. Suppose the dual problem \eqref{eqn:dual} satisfies the regularity assumption \eqref{eqn:reg} with $t>1$. Then we have the following $L^{\infty}$ error estimate:
\begin{equation}
\label{eqn:linfty-est}
	\|u - u_{h}\|_{0,\infty} \lesssim \left(h^{s-d/2} + h^{s+t-2} + h^{2(s-1)} \right) \|u\|_{s,2}.
\end{equation}
Moreover, if $s>d/2$ then for sufficiently small $h$ we have 
\begin{equation}
\label{eqn:linfty-disc}
	\|u_{h}\|_{\infty, \Omega} \leqs B
\end{equation}
for some constant $B$ independent of $h.$
\end{corollary}
\begin{proof}
	The inequality \eqref{eqn:linfty-est} follows by Lemma~\ref{lm:Linfty}, \eqref{eqn:linfty0} and \eqref{eqn:dualest}. Notice that  $\min\{ s-d/2, s+t-2, 2(s-1)\}>0$ when $s>d/2$ and $ t>1,$ then a triangle inequality yields that 
	$$
		\|u_{h}\|_{0,\infty} \leqs \|u\|_{0, \infty} + \|u - u_{h}\|_{0,\infty}.
	$$
	The inequality \eqref{eqn:linfty-disc} then follows by Theorem~\ref{thm:cont-infty} and \eqref{eqn:linfty-est} for sufficiently small $h.$
\end{proof}

\begin{remark}
	In the case of second order linear elliptic PDE, the $L^{\infty}$ error estimate in Corollary ~\ref{cor:linfty} has been discussed extensively in the literature; just a small sample includes~\cite{Bramble.J;Nitsche.J;Schatz.A1975,Douglas.J;Dupont.T;Wahlbin.L1975,Scott.R1976,Schatz.A;Wahlbin.L1977,Ciarlet.P1978,Schatz.A;Wahlbin.L1978,Schatz.A1980,Rannacher.R;Scott.R1982,Schatz.A;Wahlbin.L1995}.
    We note that in the linear case, a better rate could be achieved by using more complicated techniques, as in the aforementioned works.  
\end{remark}

\section{Numerical Examples}
   \label{sec:num}

In this section we perform some numerical experiments in two and three
dimensions to examine how the theory above is reflected in practice.
In particular, we compute solutions to
\begin{equation}
\label{eqn:num-model}
-\Delta u + u^p = f
\end{equation}
subject to homogeneous boundary conditions, where $f$ is chosen so that
the solution is known and smooth.
When $d=2$, we choose $p = 11$ in \eqref{eqn:num-model}; and when $d=3$, we choose $p$ as the critical
exponent $p = 5$.
The solutions are computed as a piecewise linear function on a
series of meshes which are uniform refinements of one of two initial
meshes; one mesh with good quality simplices and the other with large
and small angles.
The convergence profiles between the two sequences of solutions are
compared in both $H^1$ and $L^\infty$ norms.

For the example in two dimensions, the initial meshes are shown in
Figure~\ref{fig:two_dim_init} while the three-dimensional meshes are shown
in Figure~\ref{fig:three_dim_init}.  For these meshes, we compute the
triangle and tetrahedron shape metrics given by
Knupp~\cite{Knupp_2003} which gives a number between 0 and 1 quantifying the quality of the triangulation (with 1 given for isosceles simplices and 0 for
degenerate ones). These qualities are given in Table~\ref{tab:shape_qualities}.
\begin{figure}[ht]
\includegraphics[width=2in]{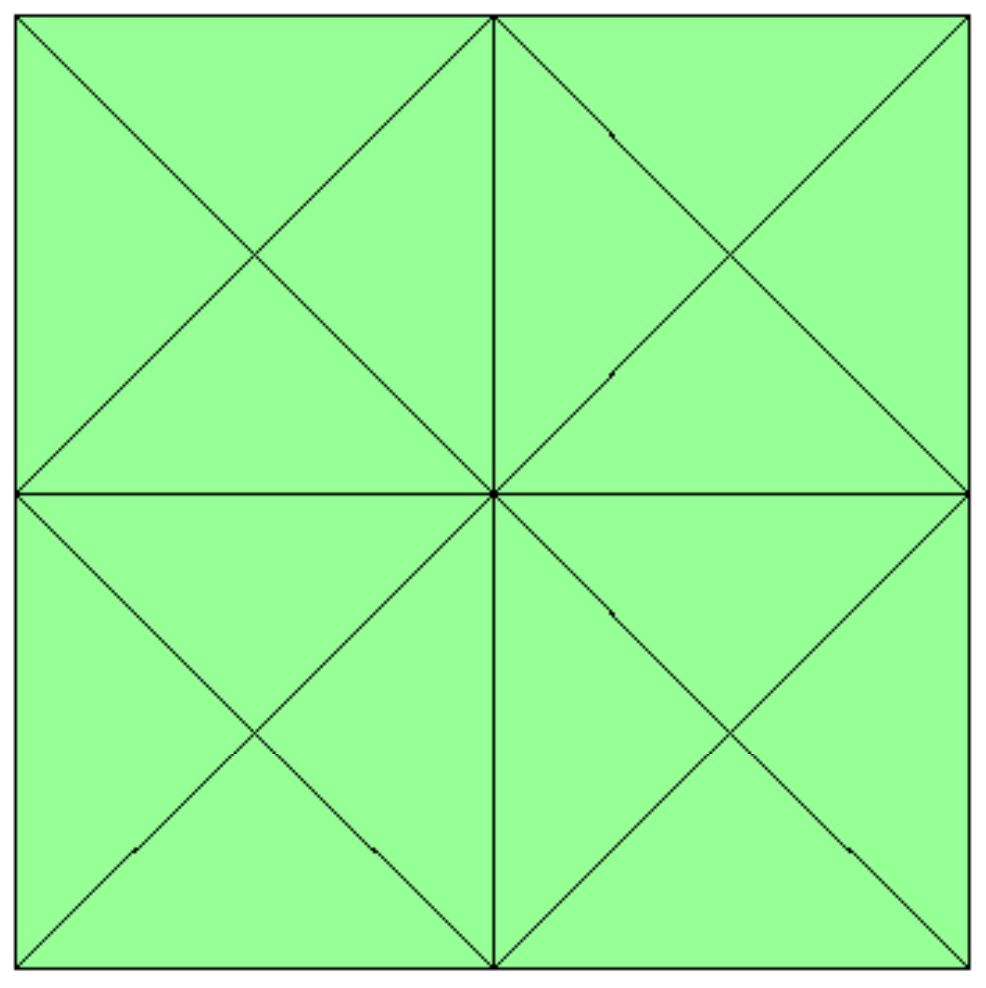}
\includegraphics[width=2in]{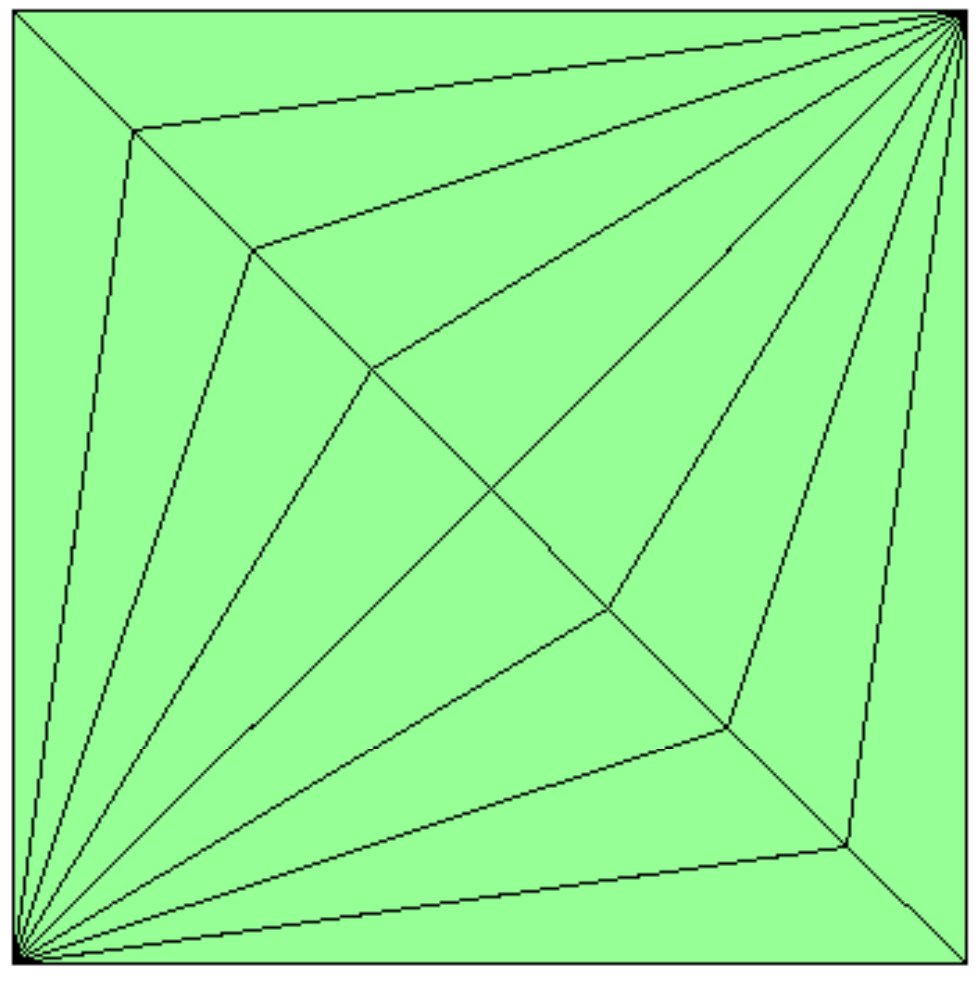}
\caption{Good and poor quality initial meshes used in the two
 dimensional examples.  The poor quality mesh has a largest angle of
 approximately $126^\circ$ and smallest angle of approximately
 $8^\circ$. \label{fig:two_dim_init}}
\end{figure}
\begin{figure}[ht]
\includegraphics[width=2in]{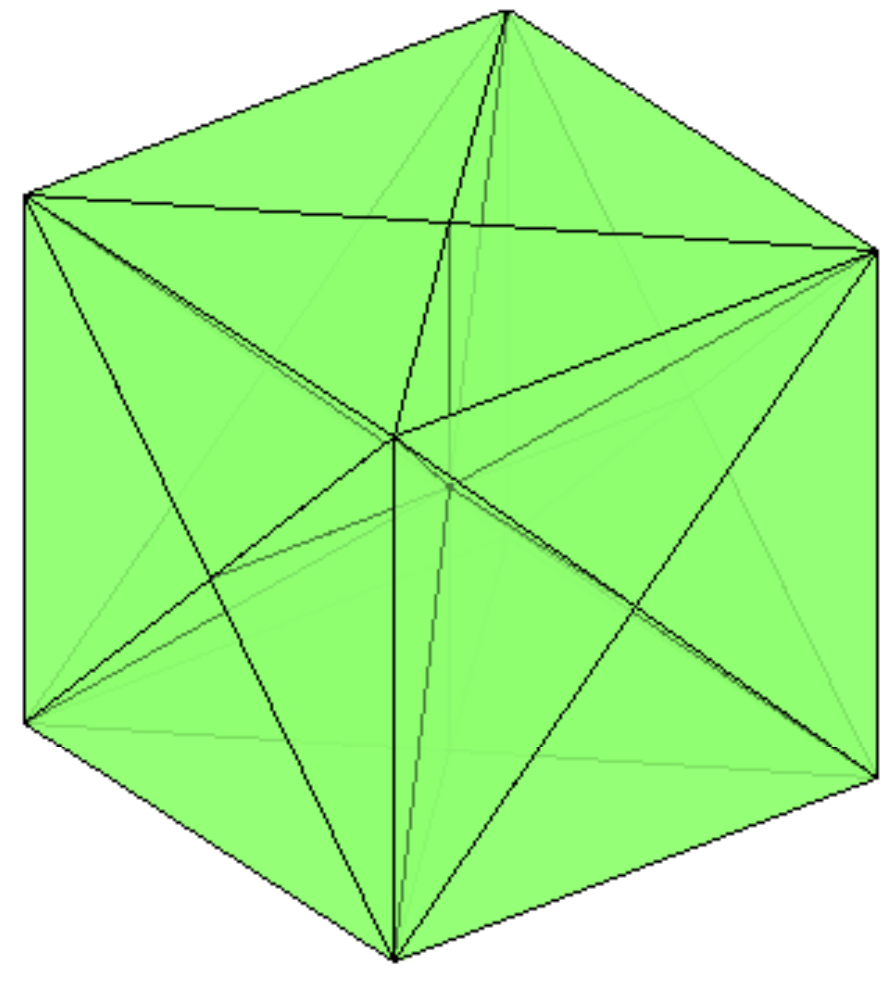}
\includegraphics[width=2in]{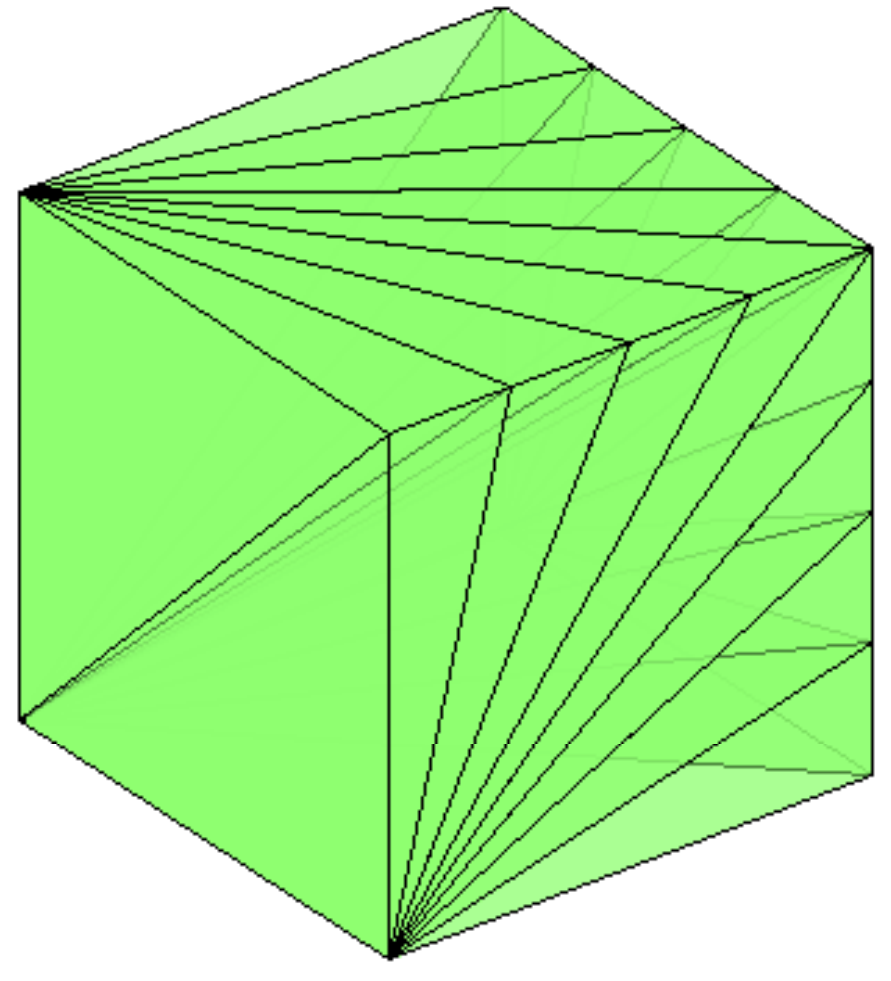}
\caption{Good and poor quality initial meshes used in the three
 dimensional examples.  The poor quality mesh has a largest dihedral
 angle of approximately $115^\circ$ and smallest dihedral angle of
 approximately $8^\circ$. \label{fig:three_dim_init}}
\end{figure}

\begin{table}[ht]
\begin{tabular}{|c|c|c|}\hline
Mesh & Worst Quality & Best Quality \\\hline
Good 2D & 0.495 & 0.693 \\\hline
Poor 2D & 0.213 & 0.283 \\\hline
Good 3D & 0.387 & 0.632 \\\hline
Poor 3D & 0.156 & 0.417 \\\hline
\end{tabular}
\vspace*{0.5\baselineskip}
\caption{Shape Quality Metrics for the various meshes. \label{tab:shape_qualities}}
\end{table}

Figure~\ref{fig:2d_error} and Figure~\ref{fig:3d_error} give the convergence results in both $H^{1}$-norm and $L^{\infty}$-norm for the semilinear problems \eqref{eqn:num-model} in 2D and 3D, respectively. We observe that the quality of the mesh does not ruin the convergence rates in 2D, as long as we keep the mesh to be quasi-uniform. On the other hand, we do observe a little deterioration on the convergence rate in $L^{\infty}$-norm in the 3D example (see Figure~\ref{fig:3d_error}) when we use a poor quality mesh. However, the errors in $H^{1}$-norm  in both 2D and 3D examples seem to be still quasi-optimal as predicted in Theorem~\ref{thm:err}.  These results confirm our theoretical conclusions. We also observe that the convergence rate in $L^{\infty}$-norm is close to $h^{2}$ in both the 2D and 3D examples, which indicates that the estimate in \eqref{eqn:linfty-est} is not optimal. Even though it is not optimal, we still got the a priori $L^{\infty}$ bound of the discrete solution in \eqref{eqn:linfty-disc}, which is important in the analysis of finite element approximation of nonlinear PDE.

\begin{figure}[ht]
\includegraphics[width=2.8in]{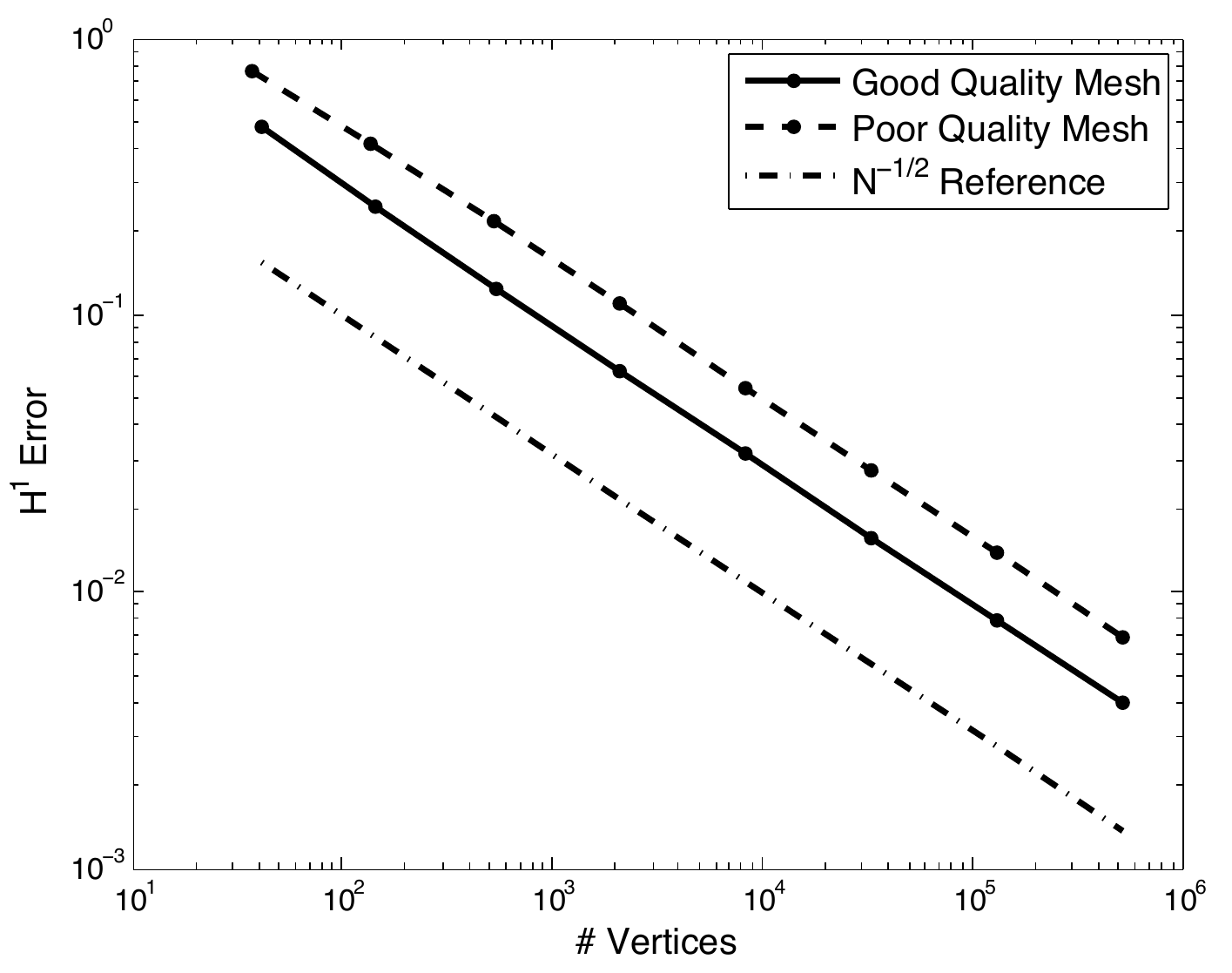}
\includegraphics[width=2.8in]{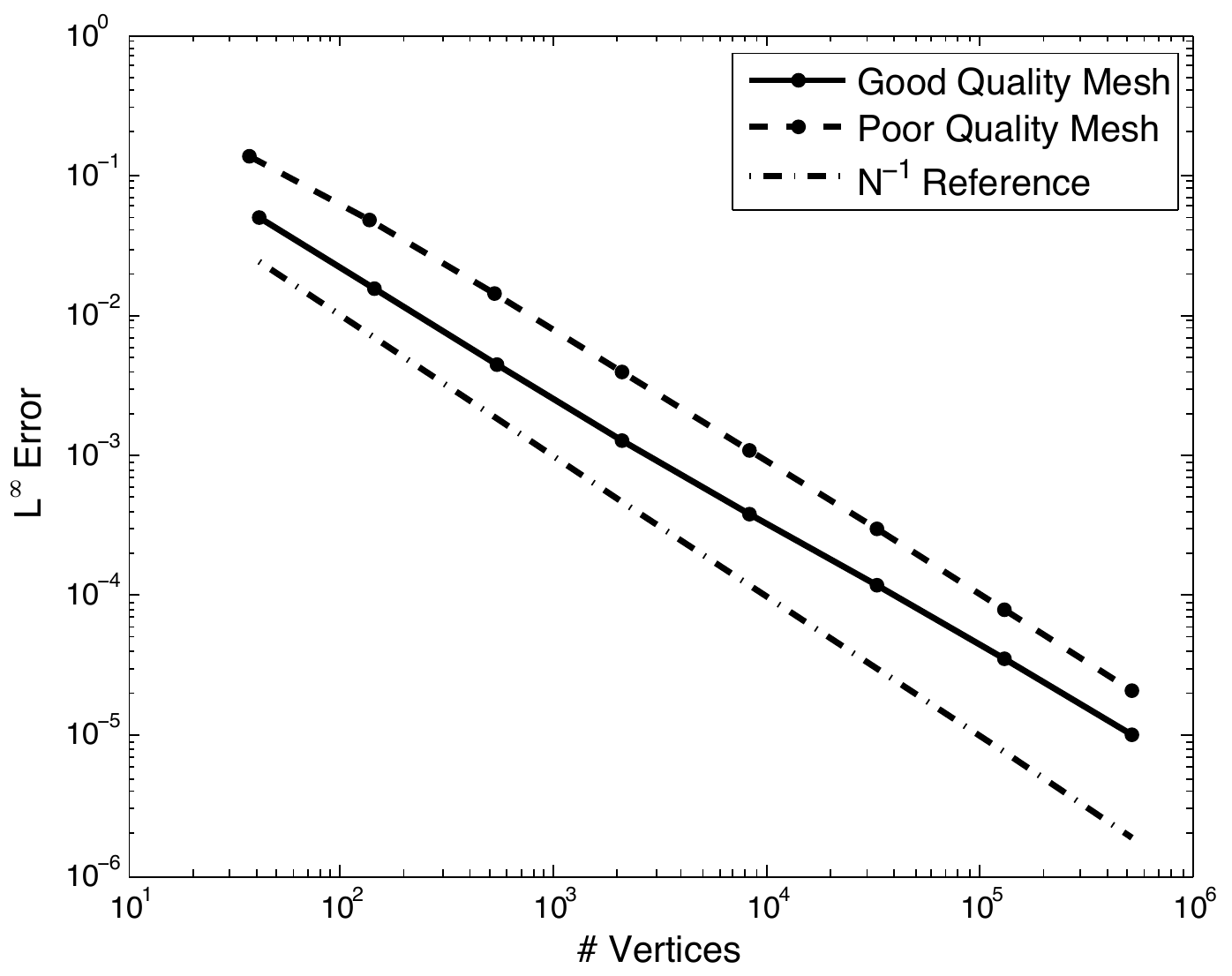}
\caption{Convergence in $H^1$ and $L^\infty$ norms for the 
two-dimensional example problem. \label{fig:2d_error}}
\end{figure}

\begin{figure}[htp]
\includegraphics[width=2.8in]{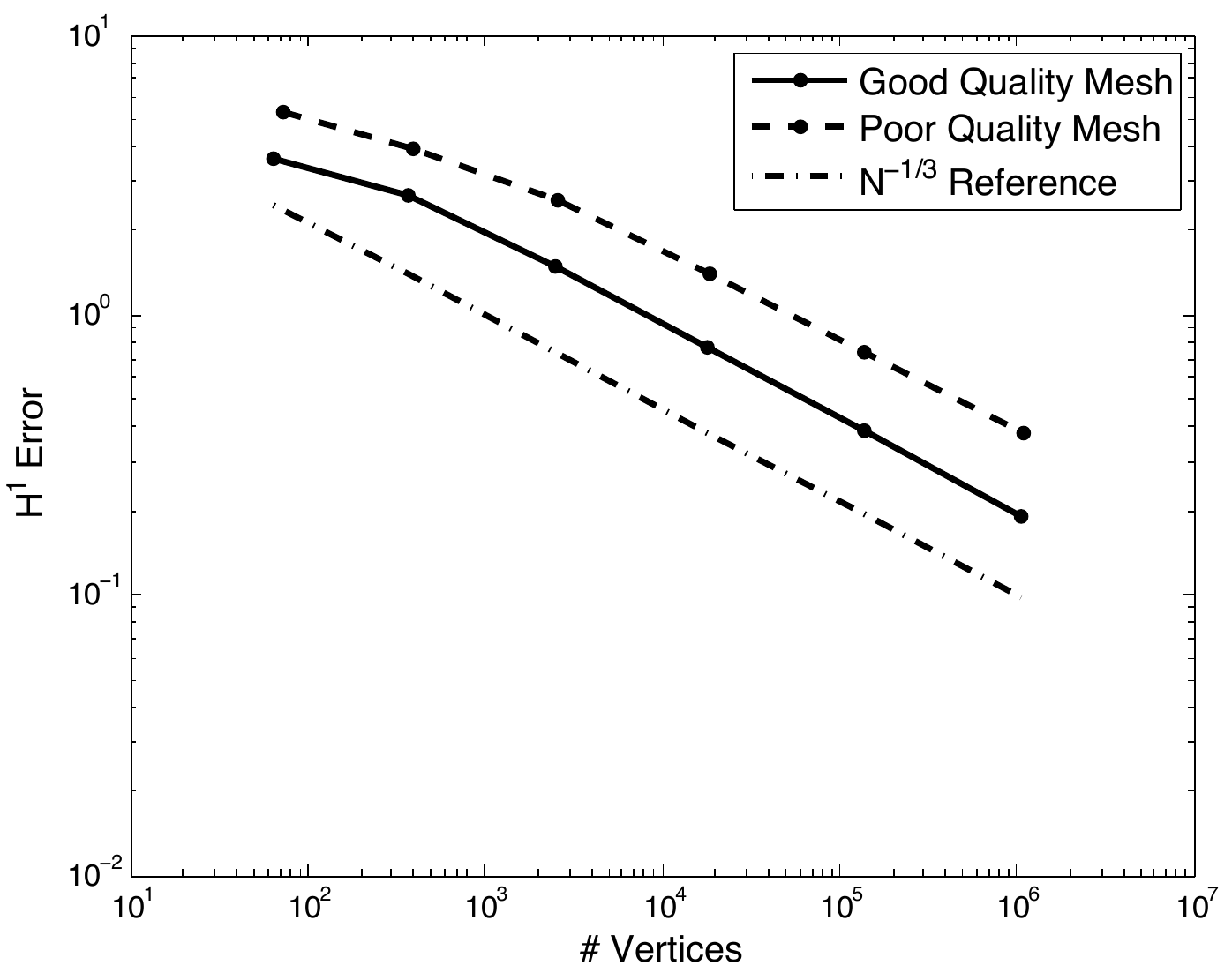}
\includegraphics[width=2.8in]{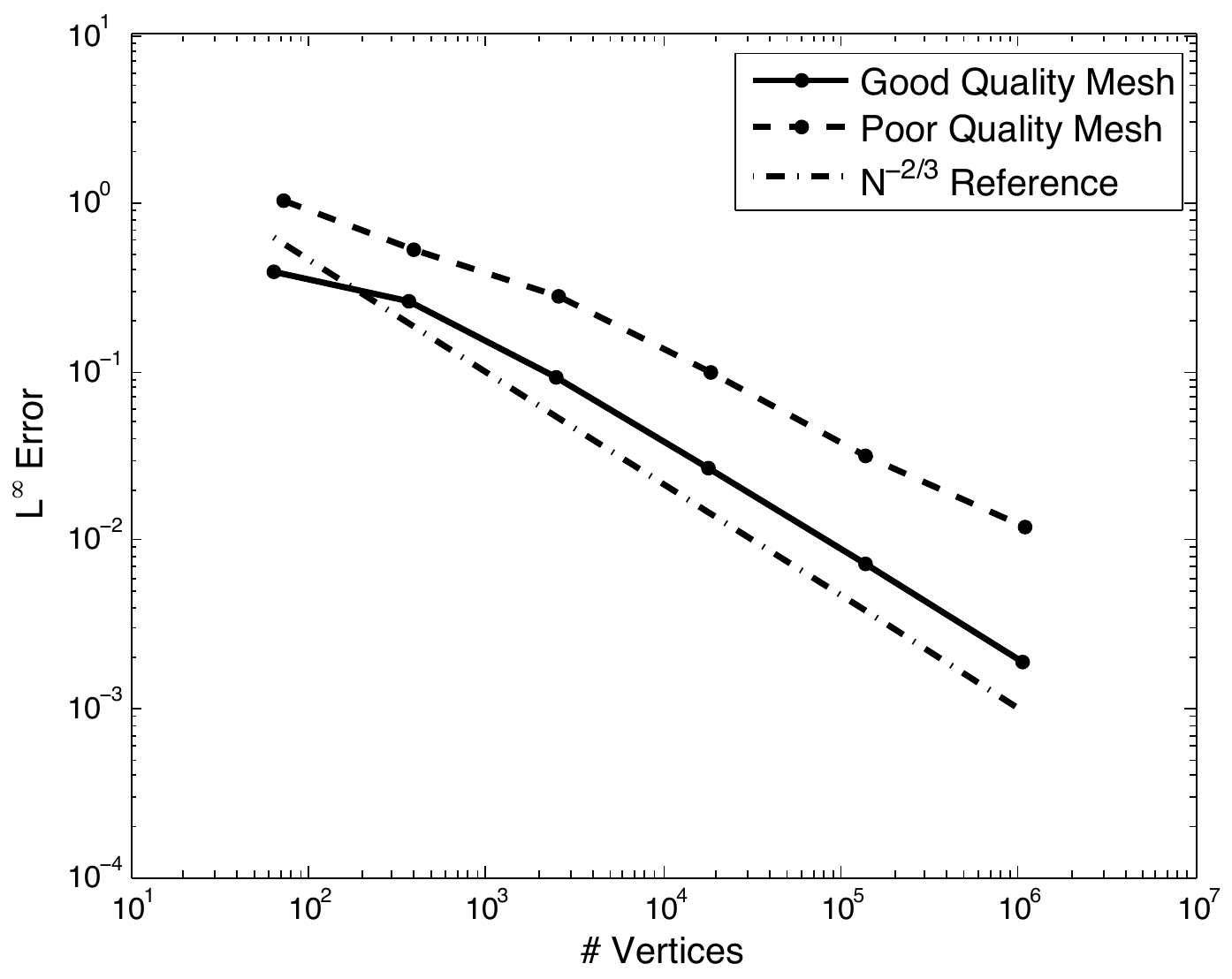}
\caption{Convergence in $H^1$ and $L^\infty$ norms for the
three-dimensional example problem. \label{fig:3d_error}}
\end{figure}

We also run a separate set of experiments in 2D in order to study \eqref{eqn:linfty-disc} in Corollary~\ref{cor:linfty}.
Specifically, starting with an initial good quality mesh, we compute the discrete solutions on successively worse meshes (produced using shortest edge bisections). Although we do not expect that the discrete solutions converge to the exact solution, we do hope that the $L^\infty$ norm of the discrete solution remains bounded.
Figure~\ref{fig:linf_norm} shows the $L^\infty$ norm of the discrete solution plotted against the size of the smallest angle in the mesh. This result confirms that the $L^{\infty}$ norm of discrete solutions are uniformly bounded as predicted in Corollary~\ref{cor:linfty}.
\begin{figure}[h]
\includegraphics[width=2.8in]{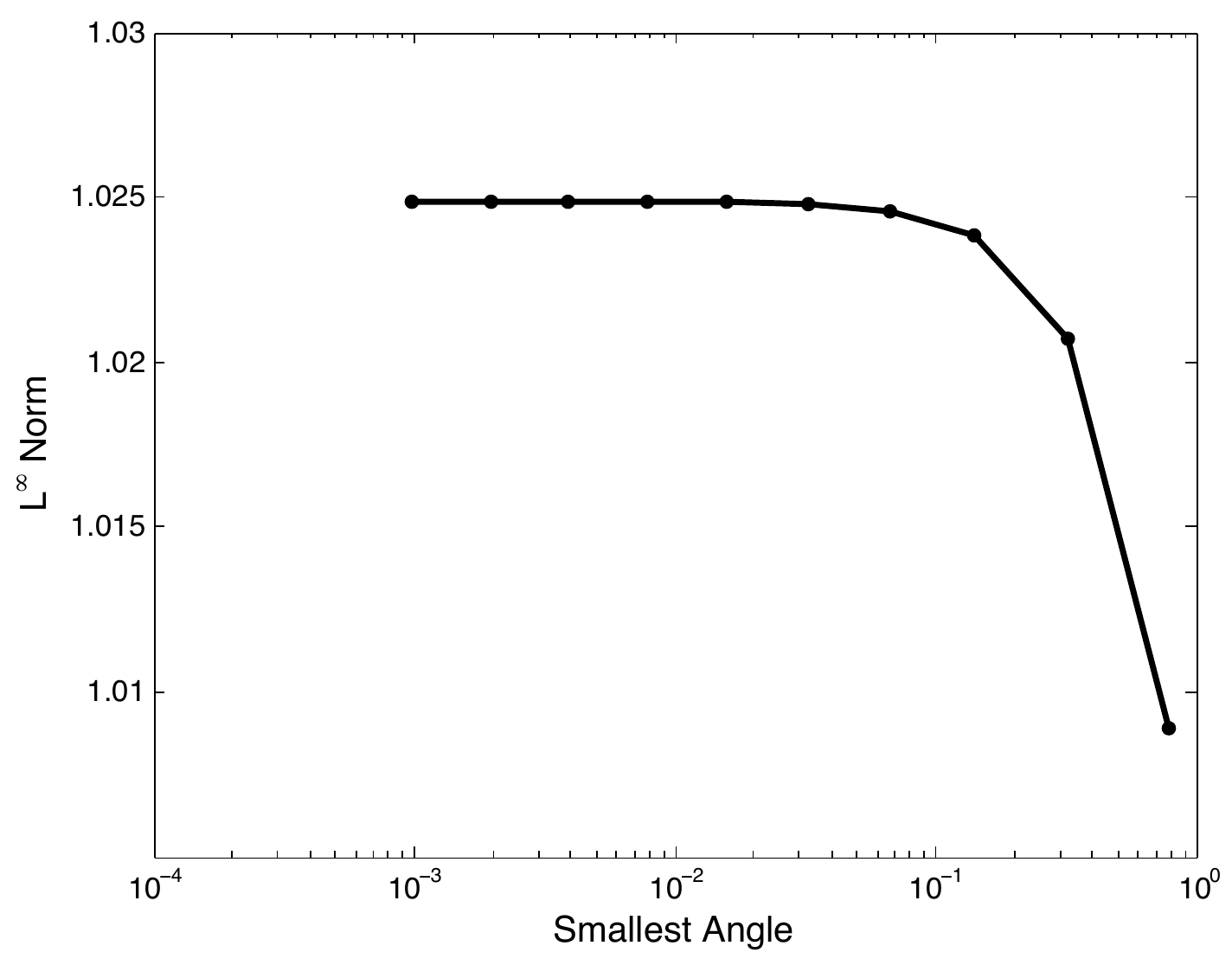}
\caption{Boundedness of the $L^\infty$ norm of the discrete solution
  computed on successively worse meshes. \label{fig:linf_norm}}
\end{figure}

\section{Conclusion}
\label{sec:conc}

In this article we considered {\em a priori} error estimates for a class of 
semilinear problems with certain growth condition, which includes problems with critical and subcritical polynomial nonlinearity in $d$ space dimensions.
Our motivation was that, while it is well-understood how mesh geometry
impacts finite element interpolant quality (at least for $d=2$ and $d=3$).
much more restrictive conditions on angles are needed to derive basic 
{\em a priori} quasi-optimal error estimates as well as {\em a priori} 
pointwise estimates for Galerkin approximations.
These angle conditions, which are particularly difficult to satisfy in 
three dimensions in any type of unstructured or adaptive setting,
are needed in order to gain pointwise control of the nonlinearity
through discrete maximum/minimum principles.
Our goal in the article was to show how to derive these types of
{\em a priori} estimates without requiring the discrete maximum/minimum principles,
hence eliminating the need for restrictive angle conditions.

To this end, in Section~\ref{sec:pde} we described a class of semilinear 
problems, 
 and reviewed the {\em a priori} $L^{\infty}$ bounds of the continuous solution through maximum/minimum principles using the De Giorgi iterative method (or Stampacchia truncation method).
We then developed a basic quasi-optimal {\em a priori} error estimate
for Galerkin approximations in Section~\ref{sec:fem}, where the nonlinearity 
was controlled by using only a local Lipschitz property rather than through 
pointwise control of the discrete solution. In this way, we avoid of using discrete maximum principle, which requires certain angle conditions. 
In particular, we showed that the local Lipschitz property in fact holds for nonlinearities satisfying certain growth condition, which includes the critical exponent cases. 
We then used some well-known results in finite element approximation theory
in Section~\ref{sec:disc-infty} to show that (under some minimal smoothness
assumptions) that the {\em a priori} error estimate is itself enough to give
$L^{\infty}$ control the discrete solution, without the need for restrictive
angle conditions that would be required to obtain a discrete maximum principle.


\section{Acknowledgments}
   \label{sec:ack}

MH was supported in part by NSF Awards~0715146 and 0915220,
and by DOD/DTRA Award HDTRA-09-1-0036.
RB was supported in part by NSF Award~0915220.
RS and YZ were supported in part by NSF Award~0715146.

\bibliographystyle{abbrv}
\bibliography{../bib/books,../bib/papers,../bib/mjh,../bib/library,../bib/ref-gn,../bib/coupling,../bib/pnp}


\vspace*{0.1cm}

\end{document}